\documentclass{amsart}
\usepackage{graphicx} 
\usepackage{hyperref}
\usepackage{fullpage}
\usepackage{cancel}
\usepackage{caption}
\usepackage{subcaption}
\usepackage{mathtools}
\usepackage{placeins}
\usepackage{amssymb}
\usepackage[normalem]{ulem}
\usepackage{thmtools} 
\usepackage{thm-restate}

\usepackage{charter}
\usepackage{euler}
\usepackage[T1]{fontenc}

\newcommand{\squareicon}[1]{\tikz{\node[rectangle, fill=#1, inner sep=5pt] {};}}
\newcommand{\circleicon}[1]{\tikz{\node[circle, fill=#1, inner sep=4pt] {};}}

\newcommand{\stariconnew}[1]{\tikz{\node[star, star points=5, fill=#1, inner sep=3pt] {};}}

\usepackage{amsmath}
\usepackage{xcolor}
\usepackage{array} 
\usepackage{tikz}
\usetikzlibrary{tikzmark}
\usetikzlibrary{decorations.pathreplacing} 
\usetikzlibrary{shapes.geometric} 
\definecolor{myred}{RGB}{220, 20, 60}
\definecolor{myblue}{RGB}{65, 105, 225}
\definecolor{mypurple}{RGB}{148, 0, 211}
\definecolor{mygreen}{RGB}{50, 205, 50}
\definecolor{myorange}{RGB}{255, 165, 0}
\definecolor{mydarkgreen}{RGB}{0, 100, 0}%

\newtheorem{theorem}{Theorem}[section]
\newtheorem*{theorem-main}{Main Theorem}

\newtheorem{lemma}[theorem]{Lemma}

\newtheorem{definition}[theorem]{Definition}
\newtheorem{proposition}[theorem]{Proposition}
\newtheorem{example}[theorem]{Example}
\newtheorem{remark}[theorem]{Remark}

\newtheorem{question}[theorem]{Question}

\newcommand{\N}{\mathbb{N}}

\newcommand{\R}{\ensuremath{\mathbb{R}}}

\newcommand{\Z}{\ensuremath{\mathbb{Z}}}

\newcommand{\conv}{\mathrm{conv}}

\newcommand{\dis}{\mathrm{dis}}
\newcommand{\codis}{\mathrm{codis}}
\newcommand{\cov}{\mathrm{cov}}

\newcommand{\diam}{\mathrm{diam}}

\newcommand{\cech}[2]{\mathrm{\check{C}}(#1;#2)}

\newcommand{\gh}{\mathrm{GH}}

\newcommand{\h}{\mathrm{H}}

\newcommand{\graph}{\mathrm{graph}}
\newcommand{\half}{\mathrm{half}}

\title{Gromov--Hausdorff Distance Between Euclidean Unit Balls}

\author{Henry Adams}
\address{Department of Mathematics, University of Florida}
\email{henry.adams@ufl.edu}

\author{Kushagri Sharma}
\address{Department of Mathematics, University of Florida}
\email{kushagrisharma@ufl.edu}

\begin{document}

\begin{abstract}
What is the Gromov--Hausdorff distance between Euclidean unit balls of different dimensions, denoted by $d_\gh(B^m,B^n)$, for $m>n$?
Note that the lower bound coming from the stability of persistent homology is zero, since all balls possess identical (trivial) persistent homology.
To establish non-trivial lower bounds, we exploit the Borsuk--Ulam theorem.
For any $m > n \ge 1$, we prove that $d_\gh(B^m,B^n)\ge \tfrac{1}{1+\alpha_n} > \frac{1}{2}$, where we determine $\alpha_n$ explicitly.
We also prove that $d_\gh(B^m,B^n)\to 1$ as $m\to \infty$ and that $d_\gh(B^m,B^n)<1$ for all finite $m>n\geq 1$.
\end{abstract}

\maketitle


\section{Introduction}

The Gromov--Hausdorff distance~\cite{edwards1975structure,gromov1981groups, gromov1981structures,memoli2007use,tuzhilin2016invented} offers a framework for quantifying the dissimilarity between arbitrary metric spaces.
Its computation is a core problem in metric geometry and its applications, with significant effort devoted to finding effective computable lower and upper bounds.
In many cases, invariants from algebraic topology, such as those derived from persistent homology, provide non-trivial lower bounds via stability theorems~\cite{ChazalDeSilvaOudot2014,chazal2009gromov}.
For instance, this approach successfully distinguishes between spheres of different dimensions and gives positive lower bounds on the Gromov--Hausdorff distance $d_\gh(S^n,S^m)$ for $m>n\geq 0$.
For improved lower bounds on $d_\gh(S^n,S^m)$, a motivational problem for this paper, see Lim, M{\'e}moli, and Smith~\cite{lim2023gromov}, as well as~\cite{GH-BU-VR}.
The stability of persistent homology similarly distinguishes compact orientable surfaces of different genera, providing positive lower bounds on $d_\gh(M_g,M_{g'})$, where $M_g$ denotes a surface of genus $g$.
However, persistence may be insensitive to geometric differences between certain spaces, including convex spaces.
This limitation is illustrated by the following question, posed in Section~8 of~\cite{HvsGH}.

\begin{question}
\label{ques:main}
Let $B^n$ be the $n$-dimensional unit ball in $\R^n$, equipped with the Euclidean metric.
What is the Gromov--Hausdorff distance $d_\gh(B^m,B^n)$ for $m > n$?
\end{question}

Persistent homology fails to give a positive lower bound for $d_\gh(B^m,B^n)$ for $m > n$, since these convex spaces have the same trivial reduced persistent homology.
Moreover, diameter-based bounds are ineffective, as all unit balls of dimension $1$ or larger share the same diameter of $2$.


\newcounter{savedtheorem}
\setcounter{savedtheorem}{\value{theorem}}

\begingroup
  
\setcounter{theorem}{0}
\renewcommand{\thetheorem}{\arabic{theorem}}
  
\makeatletter
\@namedef{theproposition}{\thetheorem}
\@namedef{thecorollary}{\thetheorem}
\let\c@proposition\c@theorem
\let\c@corollary\c@theorem
\makeatother

\subsection*{Results overview}

Our investigation into $d_\gh(B^m,B^n)$, the Gromov--Hausdorff distance between Euclidean unit balls of different dimensions, begins by exploring a classic idea from topology: the `Invariance of Dimension' theorem, which states that $\R^m$ and $\R^n$ are not homeomorphic for $m\neq n$~\cite{munkres2000topology}.
Combined with ambient \v{C}ech complexes, this yields our first result: a dimension-independent lower bound on the distance:

\begin{restatable}{proposition}{constantlowerbound}
\label{prop:constant-lower-bound}
For every $m>n\ge 1$, we have $d_\gh(B^m,B^n) \ge \frac{1}{8}$.
\end{restatable}

While this first result confirms a positive distance, it is insensitive to the dimensions involved.
To develop a better bound, we employ the Borsuk--Ulam theorem, which states that every continuous function from an $n$-sphere into Euclidean $n$-space maps some pair of antipodal points to the same point~\cite{matousek2003using}.
This approach quantifies the metric distortion inherent in any function from a high-dimensional ball to a low-dimensional one, see also~\cite{dubins1981equidiscontinuity,lim2023gromov,GH-BU-VR,rodriguez2024gromov,leon2025}.
Our analysis gives an improved dimension-dependent lower bound:

\begin{restatable}{theorem}{maintheorem}
\label{thm:general}
For every $m>n\ge 1$, we have 
\[d_\gh(B^m,B^n) \ge \tfrac{1}{1+\alpha_n} =
\begin{cases}
\frac{\sqrt{n+2}}{\sqrt{n+2}+\sqrt n},&n\text{ even},\\
\frac{\sqrt{(n+1)(n+3)}}
{\sqrt{(n+1)(n+3)}+\sqrt{n^2+2n-1}},&n\text{ odd.}
\end{cases}\]
\end{restatable}

The quantity $\alpha_n$ is the supremum of the minimum distances between finite sets in $\R^n$ of diameter at most $1$ whose convex hulls intersect (see Definition~\ref{def:alphan}).
We note that this lower bound is strictly greater than $\frac{1}{2}$, and strictly decreases to $\frac{1}{2}$ as $n\rightarrow\infty$.

We can further relate the Gromov--Hausdorff distance to how efficiently a low-dimensional ball can be covered by a finite set of points. 
By considering the \emph{covering radius} $\cov_{B^n}(m)$ of the low-dimensional ball $B^n$ by $m$ points, we obtain a lower bound that depends explicitly on both dimensions $m$ and $n$.
This strategy follows~\cite[Lemma~5.10]{colding1996large} and~\cite{funano2008estimates,lim2023gromov}.
Notably, this perspective allows us to capture the asymptotic behavior of the distance: for fixed $n$, $\cov_{B^n}(m)\rightarrow 0$ as $m\rightarrow \infty$, driving the lower bound towards $1$.

\begin{restatable}{proposition}{covering}
\label{prop:covering}
For any $m>n\geq 1$, we have $d_\gh(B^m,B^n)\geq 1-\cov_{B^n}(m)$. 
\end{restatable}

As illustrated in Figure~\ref{fig:Prop4vsThrm2}, Theorem~\ref{thm:general} provides better lower bounds on $d_\gh(B^m,B^n)$ than Proposition~\ref{prop:covering} for $m$ up to 2 when $n=1$, for $m$ up to 8 when $n=2$, for $m$ up to at least 24 when $n=3$, for $m$ up to at least 48 when $n=4$, and for $m$ up to at least 113 when $n=5$.
Only beyond these transition points, as $m$ approaches the asymptotic regime, does Proposition~\ref{prop:covering} meet or overtake Theorem~\ref{thm:general}.

\begin{figure}[hbt]
    \centering
    \includegraphics[width=0.6\linewidth]{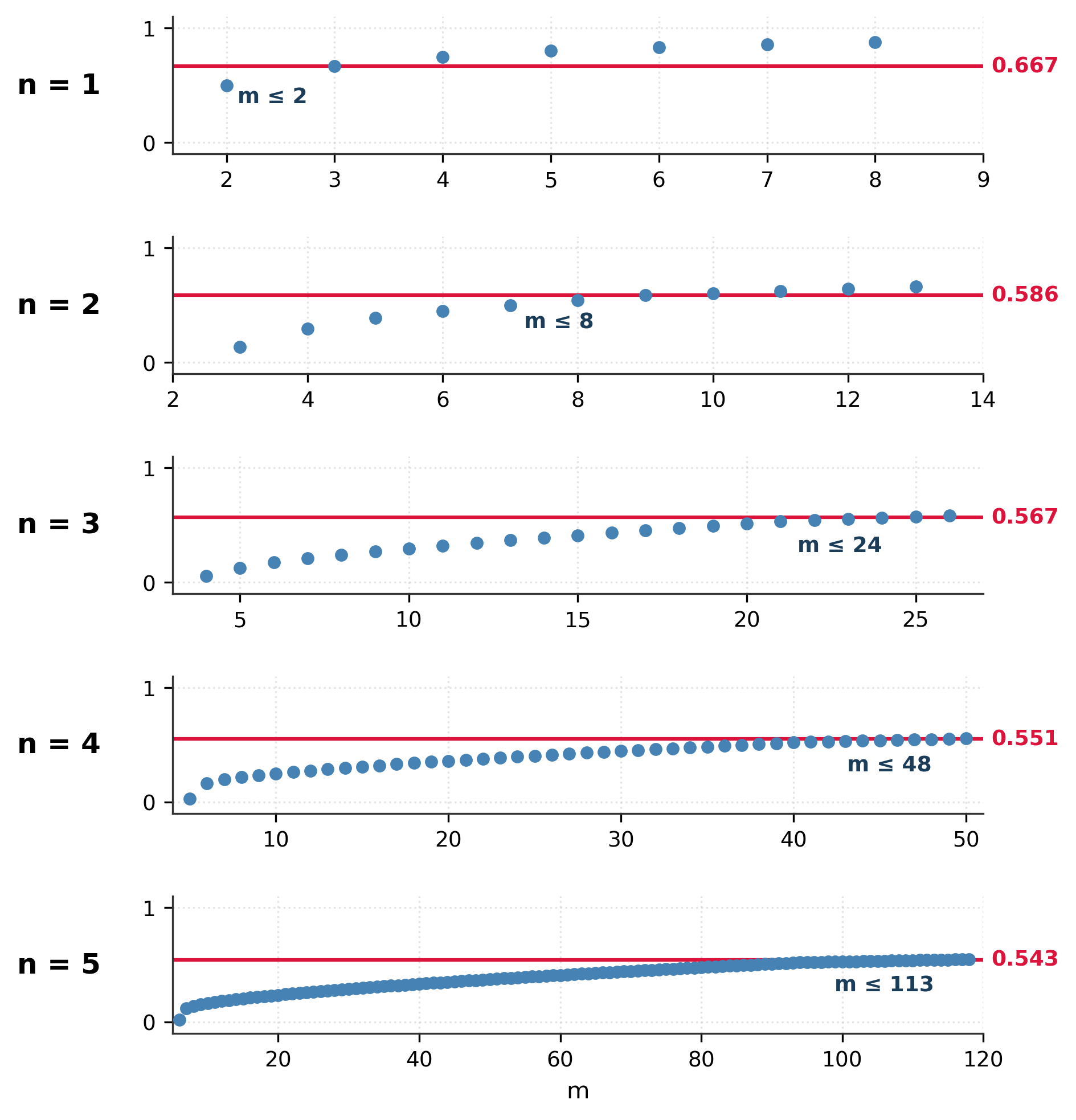}
    \caption{Plots comparing lower bounds on $d_\gh(B^m,B^n)$ from Theorem~\ref{thm:general} (horizontal red lines at $\tfrac{1}{1+\alpha_n}$) and Proposition~\ref{prop:covering} (blue dots) for dimensions $n \in \{1, 2, 3, 4, 5\}$.
    The annotated thresholds ($m\le 2$, $m\le 8$, $m\le 24$, \ldots) mark the regime where Theorem~\ref{thm:general} is known to improve upon Proposition~\ref{prop:covering}.}
    \label{fig:Prop4vsThrm2}
\end{figure}

Our narrative then shifts to the other side of the problem: the upper bounds on $d_\gh(B^m,B^n)$ for $m>n\geq 1$.
While the natural isometric embedding of the low-dimensional ball into the high-dimensional ball gives $d_\gh(B^m,B^n)\le 1$, we show that this bound is never attained in finite dimensions.
In Section~\ref{sec:upperbounds}, we prove that the Gromov--Hausdorff distance between Euclidean unit balls is always strictly less than $1$.
We do this by adapting ideas from~\cite{lim2023gromov}, constructing an odd, continuous surjective map from the low-dimensional ball to the high-dimensional one.

\begin{restatable}{theorem}{upperboundlessthanone}
\label{thrm:upperbound}
For $m>n\geq 1$, we have $d_\gh(B^m,B^n)<1$.
\end{restatable}

\endgroup

\setcounter{theorem}{\value{savedtheorem}}

Figure~\ref{fig:ball_matrix} shows a summary of some of these results in matrix form.

\begin{figure}[thbp] 
\centering
\hspace*{-6mm}
\setlength{\arraycolsep}{8pt} 
\renewcommand{\arraystretch}{1.8} 
\begin{tabular}{c c}
\(
\left[
\begin{array}{c|c|c|c|c|c|c|c|c|c}
    \hline
      & B^0 & B^1 & B^2 & B^3 & B^4 & \dots & \tikzmarknode{brace-start}{B^{m}} & \dots & \tikzmarknode{brace-end}{B^{\infty}} \\
    \hline
    B^0 & \circleicon{mygreen} & \squareicon{myorange} & \squareicon{myorange} & \squareicon{myorange} & \squareicon{myorange} & \squareicon{myorange} & \squareicon{myorange} & \squareicon{myorange} & \squareicon{myorange} \\
    \hline
    B^1 &  & \circleicon{mygreen} & \stariconnew{mypurple} & \stariconnew{mypurple} &\stariconnew{mypurple} &\stariconnew{mypurple} & \stariconnew{mypurple}&\stariconnew{mypurple} & \squareicon{myorange} \\
    \hline
    B^2 & & & \circleicon{mygreen} & \stariconnew{mypurple} & \stariconnew{mypurple} & \stariconnew{mypurple}& \stariconnew{mypurple} & \stariconnew{mypurple}& \squareicon{myorange} \\
    \hline
    B^3 & & & & \circleicon{mygreen} & \stariconnew{mypurple} &\stariconnew{mypurple} &\stariconnew{mypurple} & \stariconnew{mypurple} & \squareicon{myorange} \\
    \hline
    B^4 & & & & & \circleicon{mygreen} & \stariconnew{mypurple} &\stariconnew{mypurple}& \stariconnew{mypurple}& \squareicon{myorange} \\
    \hline
    \vdots & & & & & & \circleicon{mygreen} & \stariconnew{mypurple}&\stariconnew{mypurple} & \squareicon{myorange} \\
    \hline
    B^{n} & & & & & & & \circleicon{mygreen} & \stariconnew{mypurple}& \squareicon{myorange} \\
    \hline
    \vdots & & & & & & & &   \circleicon{mygreen} & \squareicon{myorange} \\
    \hline 
    B^\infty & & & & & & & & & \circleicon{mygreen} \\
    \hline
\end{array}
\right]
\)
&

\hspace{2em}

\begin{tabular}{l l}
\squareicon{myorange}   & 1 \\
\stariconnew{mypurple}  & $\big[\max\big\{\tfrac{1}{1+\alpha_n}, 1-\cov_{B^n}(m)\big\},1\big)$ \\
\circleicon{mygreen}    & $0$ \\
\end{tabular}

\end{tabular}

\caption{
Matrix of Gromov--Hausdorff distances $d_\gh(B^m,B^n)$: each entry gives either the exact value or an interval containing it.
For an entry in row $B^n$ and column $B^m$ with $m>n$, the $n$ in the interval formula denotes the smaller dimension.
}
\label{fig:ball_matrix} 
\end{figure}

Our paper is structured as follows.
In Section~\ref{sec:preliminaries}, we provide preliminary definitions, notation, and lemmas.
Section~\ref{sec:invOFdim} adapts the classic topological proof of `Invariance of Dimension', which shows that the Euclidean spaces $\R^n$ and $\R^m$ cannot be homeomorphic for $m>n$, to lower bound $d_\gh(B^m,B^n)$.
Section~\ref{secBorsukUlam} employs the Borsuk--Ulam theorem to prove Theorem~\ref{thm:general} and the covering radius to prove Proposition~\ref{prop:covering}.
Section~\ref{sec:upperbounds} explores upper bounds on the Gromov--Hausdorff distance between Euclidean balls of different dimensions.
To conclude in Section~\ref{sec:conclusion}, we list some open questions.
Appendix~\ref{app:proofoflemmageneral} contains the proof of the lemma giving the exact value of $\alpha_n$, Appendix~\ref{app:oddsctssurjective} constructs the surjections needed to prove Theorem~\ref{thrm:upperbound}, and Appendix~\ref{sec:balls-curvature-sets} compares bounds obtained from curvature sets.

\subsection*{AI disclosure}
We used ChatGPT-5.6 Sol in Appendix~\ref{app:proofoflemmageneral} to prove Lemma~\ref{lem:algebraic-general}.
We take responsibility for all of the mathematics in the paper.

\subsection*{Acknowledgements}
This research was funded by the NSF CAREER Grant DMS 2540172.
The first author would also like to thank the Simons Foundation's Travel Support for Mathematicians.

\section{Preliminaries, notation, and background lemmas}
\label{sec:preliminaries}

We begin with preliminaries on Hausdorff and Gromov--Hausdorff distances, distortion and codistortion, correspondences, covering radii, \v{C}ech simplicial complexes, and triangulations.

\subsection{Metric spaces}

Let $(X,d)$ be a metric space.
For any $x\in X$ and $r>0$, we let $B(x,r)=\{x' \in X \mid d(x,x')<r\}$ denote the open metric ball of radius $r$ centered at $x$.
For $X'\subseteq X$, we let $B(X',r)=\bigcup\limits_{x\in X'}B(x,r)$ be the $r$-neighborhood of $X'$.

The \emph{diameter} of a subset $A\subseteq X$ is defined as the supremum of all pairwise distances between points in $A$:
\[\diam(A)=\sup\{d(x,x')\mid x,x'\in A\}.\]
If $A$ is compact, its diameter is finite and the supremum is attained.

Let $d(a,B)=\inf\limits_{b\in B}d(a,b)$ denote the distance from a point $a\in X$ to a subset $B\subseteq X$.
The distance between two non-empty subsets $A,B\subseteq X$ is given by the infimum of pairwise distances between their elements:
\[d(A,B)=\inf\{d(a,b)\mid a\in A, b\in B\}.\]
Note that if $A$ and $B$ intersect ($A\cap B\neq \emptyset$), then $d(A,B)=0$.

Given $\varepsilon>0$, an \emph{$\varepsilon$-net} for a metric space $X$ is a subset $A\subseteq X$ such that every point in $X$ lies within distance $\varepsilon$ of $A$; that is, for every $x\in X$, there exists $a\in A$ with $d_X(x,a)\leq \varepsilon$.

For a set $X\subseteq \R^n$, we denote its convex hull by $\conv(X)$, the smallest convex set containing $X$.

Our analysis relies on Jung's theorem, which bounds the radius of the smallest ball enclosing any bounded set in Euclidean space:

\begin{theorem}[Jung's Theorem~\cite{Jung1901}]
\label{thm:JungsTheorem}
Any bounded set $K\subseteq \R^n$ has the smallest enclosing ball of radius $R_K$ satisfying $R_K \leq \diam(K) \sqrt{\tfrac{n}{2(n+1)}}$.
Furthermore, equality is attained when $K$ is the vertex set of a regular $n$-simplex.
\end{theorem}

\subsection{Euclidean unit balls}

For any finite dimension $n\ge 0$, we consider the $n$-dimensional closed unit ball:
\[B^n\coloneqq \{(x_1,\ldots,x_n)\in \R^n \mid x_1^2+\ldots+x_n^2\leq 1\}.\]
We view $B^n$ as a metric space equipped with the standard Euclidean distance: for any two points $x,x'\in B^n$,
\[d(x,x')\coloneqq \left(\sum\limits_{i=1}^{n}(x_i-x'_i)^2\right)^{\frac{1}{2}}.\]
For $n=0$, $B^0$ consists of a single point.

For the infinite-dimensional case, we let $B^\infty$ denote the unit ball in the Hilbert space $\ell^2$, defined as $B^\infty\coloneqq \{(x_i)_{i=1}^{\infty}\in \ell^2\mid \sum\limits_{i=1}^{\infty}x
_i^{2}\leq 1\}$, which we endow with the standard $\ell^2$ metric $d(x,x')\coloneqq \left(\sum_{i=1}^{\infty}(x_i-x'_i)^2\right)^{\frac{1}{2}}$.

For $n\in \N$ we denote the $(n-1)$-dimensional unit sphere by $S^{n-1}\coloneqq \{(x_1,\ldots,x_n)\in \R^n \mid x_1^2+\ldots+x_n^2= 1\}$.
We denote the infinite-dimensional unit sphere by $S^{\infty}\coloneqq \{(x_i)_{i=1}^{\infty}\in \ell^2\mid \sum\limits_{i=1}^{\infty}x
_i^{2}= 1\}$.

\subsection{Hausdorff and Gromov--Hausdorff distances}
\label{H&GH}

\begin{definition}[Hausdorff Distance]
\label{def:hausdorff}
Let $M$ be a metric space.
For each pair of non-empty subsets $X\subseteq M$ and $Y\subseteq M$, the \emph{Hausdorff distance} between $X$ and $Y$ is defined as: 
\[d_\h(X,Y)=\max\left\{\sup\limits_{x\in X}d(x,Y), \sup\limits_{y\in Y}d(y,X)\right\}.\]
\end{definition}
When $X$ and $Y$ are compact, $d_\h(X,Y)$ represents the farthest distance any point in $X$ can be from $Y$, or vice versa, whichever is greater.
The Hausdorff distance can be infinite as well.
Figure~\ref{fig:hdistance} illustrates the Hausdorff distance on an example.

\begin{figure}[htb]
\centering
\includegraphics[width=0.5\textwidth]{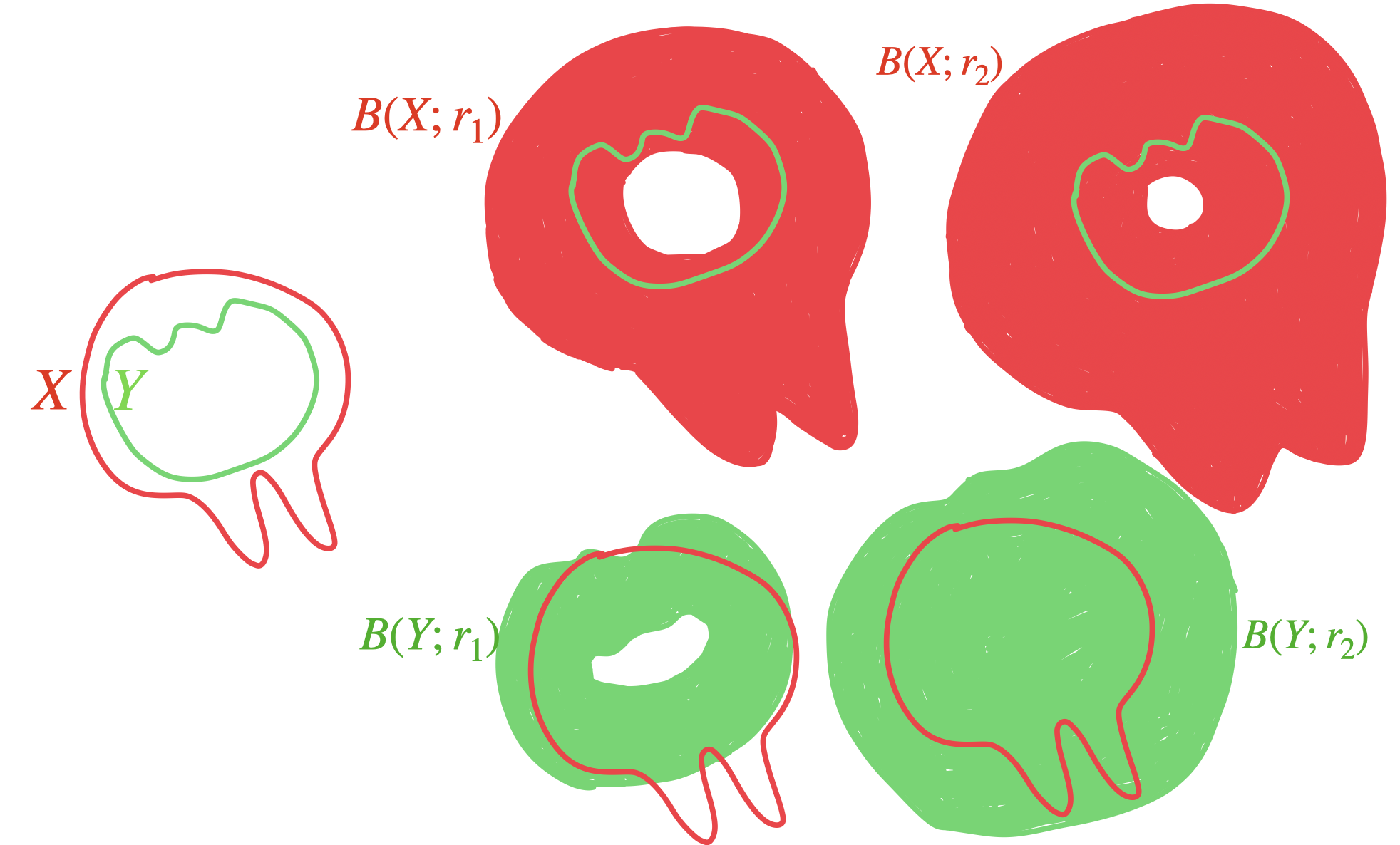}
\caption{
Metric spaces $X$ (red) and $Y$ (green) inherit the Euclidean metric
from the plane.
Observe that $Y\subseteq B(X,r_1)$ but $X \not\subseteq B(Y,r_1)$, so $d_\h(X,Y)\geq r_1$.
However, $X\subseteq B(Y,r_2)$ and $ Y\subseteq B(X,r_2)$, so  $d_\h(X,Y)\leq r_2$.}
\label{fig:hdistance}
\end{figure}

The covering radius of a metric space can be defined in terms of the Hausdorff distance.

\begin{definition}[Covering radius]
For any integer $m\geq 1$ and compact metric space $X$, the \emph{$m$-th covering radius of $X$} is defined as:
\begin{equation}
\label{eq:defcoveringradius}
\cov_X(m)=\inf\{d_\h(P,X)\mid \emptyset \neq P\subseteq X \text{ such that } |P|\leq m\}.
\end{equation}
\end{definition}

Unlike the Hausdorff distance, the Gromov--Hausdorff distance
compares metric spaces $X$
and $Y$ that are not necessarily subsets of the same metric space.

\begin{definition}[Gromov--Hausdorff distance~\cite{edwards1975structure,gromov1981groups, gromov1981structures,tuzhilin2016invented}]
\label{GHdef}
The \emph{Gromov--Hausdorff distance} $d_\gh(X,Y)$ between two metric spaces $X$ and $Y$ is defined as the infimum, over all metric spaces $Z$ and isometric embeddings $\phi\colon X\rightarrow Z$ and $\psi\colon Y\rightarrow Z$, of the Hausdorff distance in $Z$ between $\phi(X)$ and $\psi(Y)$.
\end{definition}

\begin{figure}[htb]
\centering
\includegraphics[width=0.17\textwidth]{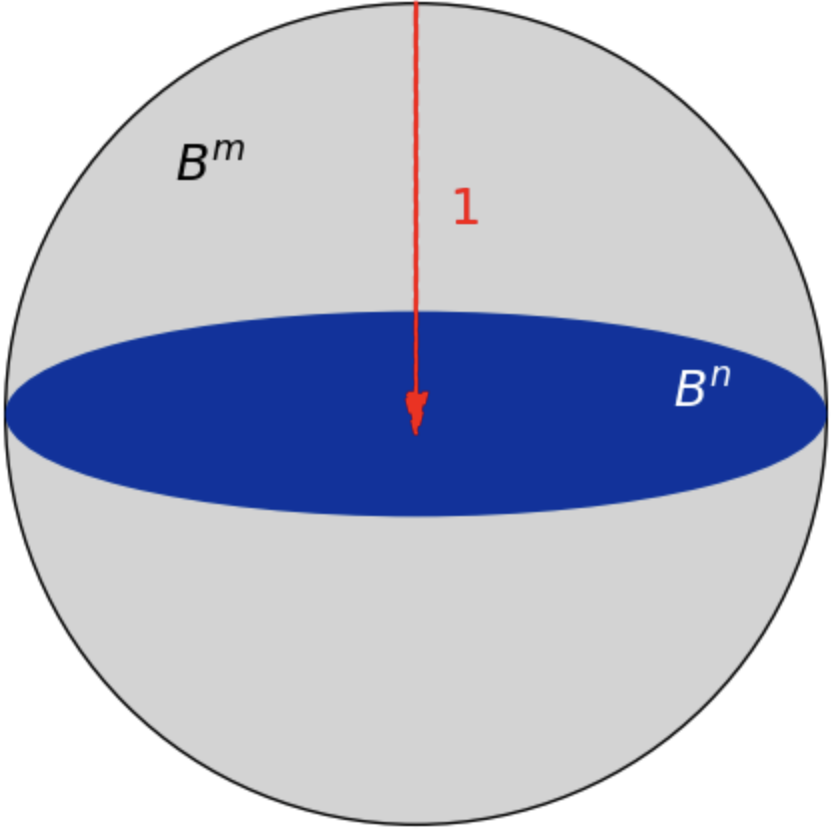}
\caption{An isometric embedding of $B^n \hookrightarrow B^m$ for $m>n$ given by $(x_1,x_2, \ldots, x_n)\mapsto (x_1,x_2, \ldots, x_n,0,\ldots,0)$.
This embedding gives $d_\gh(B^m,B^n)\leq 1$.
}

\end{figure}
It is known that $d_\gh$ defines a metric on compact metric spaces up to isometry.
To see that one can avoid taking an infimum over a proper class, restrict attention to the
case where $Z$ is the disjoint union of $X$ and $Y$, equipped with a metric extending the metrics on $X$
and $Y$.

\begin{definition}[Distortion and Codistortion]
    Define the \emph{distortion} of a function $\phi\colon X \rightarrow Y$ to be
\begin{equation}
\label{eq:dis_calc}
\dis(\phi)= \sup\limits_{x,x' \in X}\lvert d_X(x,x')-d_Y(\phi(x),\phi(x'))\rvert.
\end{equation} 
For two functions $\phi\colon X \rightarrow Y$ and $\psi\colon Y \rightarrow X$, define the \emph{codistortion} to be the coupling term:
\[\codis(\phi, \psi)=\sup\limits_{x\in X, y\in Y} \lvert d_X(x, \psi(y))-d_Y(\phi(x),y)\rvert.\]
\end{definition}
Intuitively, distortion measures how much a mapping stretches or shrinks pairwise distances within a single metric space, whereas codistortion quantifies how well a pair of forward and backward functions between two spaces work as pseudo-inverses of one another. 

In~\cite{kalton1999distances}, Kalton and Ostrovskii show that the Gromov--Hausdorff distance between compact metric spaces $X$ and $Y$ is equal to one half of an infimum of (co)distortions:
\begin{equation}
\label{eq:def-gh-dis-codis}
d_\gh(X,Y)=\frac{1}{2}\inf\limits_{\begin{smallmatrix} \phi\colon X \rightarrow Y & \\ \psi\colon Y \rightarrow X \end{smallmatrix}} \max(\dis(\phi), \dis(\psi), \codis(\phi, \psi)).
\end{equation}

\subsection{Correspondences and distortion}
\label{sec:correspondences}
Another equivalent definition of the Gromov--Hausdorff distance is given through correspondences.
A relation $C\subseteq X\times Y$ is a \emph{correspondence} if the following two
conditions hold:
\begin{enumerate}
\item For every $x\in X$, there exists $y\in Y$ such that $(x,y)\in C$, and
\item For every $y\in Y$, there exists $x\in X$ such that $(x,y)\in C$.
\end{enumerate}
So, every point in $X$ is related to at least one point in $Y$, and vice-versa.
The \emph{distortion} of a correspondence $C$
between $(X,d_X)$ and $(Y,d_Y)$ is defined as:
\[\dis(C)=\sup\limits_{(x,y), (x',y')\in C}\lvert d_X(x,x')-d_Y(y,y')\rvert.\]
We can equivalently define the Gromov--Hausdorff distance between two metric spaces $X$ and $Y$ as
\begin{equation}
\label{eq:GHandCorres}
d_\gh(X,Y)=\frac{1}{2}\inf\limits_{C\subseteq X\times Y}\dis(C),
\end{equation}
where the infimum is taken over all correspondences $C$ between $X$ and $Y$~\cite{BuragoBuragoIvanov,kalton1999distances}.

If $X$ and $Y$ are both compact metric spaces, then the infimum in the definition of the Gromov--Hausdorff distance is attained~\cite{IvanovIliadisTuzhili}.

\subsection{Bounds using diameter}

For any two bounded metric spaces $X$ and $Y$, the Gromov--Hausdorff distance is bounded by their diameters:
\begin{equation}
\label{eq:lessthan1forall}
d_\gh(X,Y)\leq \tfrac{1}{2}\max\{\diam(X),\diam(Y)\}.
\end{equation}
In order to prove this well-known result, let $p_X\in X$ and $p_Y\in Y$ be arbitrary points. 
Consider the correspondence $C=(X\times \{p_Y\})\cup (\{p_X\}\times Y)$.
The distortion of this correspondence is $\dis(C)=\max\{\diam(X),\diam(Y)\}$, since for any two points $x_1,x_2 \in X$, the pair $(x_1,p_Y)$ and $(x_2,p_Y)$ contributes $\lvert d_X(x_1,x_2)-0\rvert$ to the distortion.
Finally,~\eqref{eq:GHandCorres} proves the desired result.

Consequently, $d_\gh(B^m,B^n)\leq 1$ for all $m> n\geq 1$.
When one of the spaces is $B^0$, the single point metric space, we have $d_\gh(B^m, B^0)=\frac{1}{2}\diam(B^m)=1$ for $m\ge 1$.

\subsection{Simplicial complexes}
The next two subsections introduce two simplicial constructions used in our proofs: \v{C}ech complexes for the `invariance of dimension' approach (Section~\ref{sec:invOFdim}), and triangulations for the Borsuk--Ulam approach (Section~\ref{secBorsukUlam}). 
We often identify simplicial complexes with their geometric realizations.

We begin with the basic properties of simplicial maps.
A simplicial map $f\colon K \rightarrow L$ is continuous on geometric realizations.
If $K$ and $L$ are simplicial complexes and $f,g\colon K \rightarrow L$ are simplicial maps, then $f$ and $g$ are \emph{contiguous} if for every simplex $\sigma \in K$, the union $f(\sigma)\cup g(\sigma)$ is a simplex in $L$.
Contiguous simplicial maps induce homotopic maps on their geometric realizations, a property we exploit in Section~\ref{sec:invOFdim}.

\subsection{\v{C}ech complexes}
\begin{definition}
\label{def:cech}
For $Z$ a metric space, $X \subseteq Z$, and $r>0$, the \emph{\v{C}ech complex} $\cech{X}{r}$ is the simplicial complex having $X$ as its vertex set, and a finite set $[x_0, x_1, \ldots, x_k]\subseteq X$ as a simplex if there is some $z\in Z$ with $d(z,x_i)< r$ for all $0\leq i\leq k$.
\end{definition}

Equivalently, $[x_0, x_1, \ldots, x_k]$ is a simplex if $\bigcap_{i=0}^{k}B(x_i,r)\neq \emptyset$.
This construction is called the \emph{ambient} \v{C}ech complex, as we are looking at the intersection of balls in $Z$.
The case where $Z=X$ is referred to as the \emph{intrinsic} \v{C}ech complex, which we do not use in this paper.

The proof of Proposition~\ref{prop:constant-lower-bound} relies on the following technical lemma relating correspondences to \v{C}ech complexes~\cite{HvsGH}:

\begin{lemma}[Maps between ambient
\v{C}ech complexes, ~\cite{ChazalDeSilvaOudot2014} and Lemma~2.1 of~\cite{HvsGH}]
\label{lem:ambientCech}
Let $Z$ be a metric space, let
$X\subseteq Z$, and let $r>2d_\gh(X,Z)$.
This allows us to pick functions $h\colon Z \to X$ and $g\colon X\to Z$ with $\dis(h),\dis(g), \codis(h,g)< r$.
Then for any $\varepsilon>0$, 
the functions $h$ and $g$ induce simplicial maps $\bar{h}$ and $\bar{g}$
\[\cech{Z}{\varepsilon}\xrightarrow{\bar{h}} \cech{X}{r+\varepsilon}\xrightarrow{\bar{g}} \cech{Z}{3r+2\varepsilon}\]
such that the composition $\bar{g}\circ \bar{h}$ is contiguous to the inclusion $i\colon\cech{Z}{\varepsilon}\hookrightarrow \cech{Z}{3r+2\varepsilon}$, where all complexes are ambient
\v{C}ech complexes using balls in $Z$.
\end{lemma}

\subsection{Triangulations}

For the Borsuk--Ulam approach in Section~\ref{secBorsukUlam}, we require the notion of triangulations with special symmetry properties.
A triangulation of a topological space $X$ is a representation of 
$X$ up to homeomorphism as the geometric realization of a simplicial complex.

\begin{definition}[Triangulation]
\label{def:triang}
A \emph{triangulation of a topological space $X$} is a pair $(\tau,t)$ where $\tau$ is a simplicial complex and $t\colon \tau \rightarrow X$ is a homeomorphism.
\end{definition}

Using $t$, we can identify the vertices of $\tau$ with points in $X$.
For metric spaces, we can impose a size constraint as follows:

\begin{definition}[$\varepsilon$-triangulation]
\label{def:ep-triang}
Let $X$ be a metric space and let $\varepsilon>0$.
A triangulation $(\tau,t)$ of $X$ is an \emph{$\varepsilon$-triangulation} if the vertex set $V(\tau)$ is a subset of $X$, if $t|_{V(\tau)}$ is the inclusion $V(\tau)\hookrightarrow X$, and if any simplex in  $\tau$ has diameter at most $\varepsilon$.
\end{definition}

\subsection{The Borsuk--Ulam theorem}
The main tool for lower bounding $d_\gh(B^m,B^n)$ used in Section~\ref{secBorsukUlam} is the Borsuk--Ulam theorem.
See~\cite{matousek2003using} for an exposition of the material in this section.

\begin{theorem}[The Borsuk--Ulam theorem~\cite{borsuk1933drei}]
Every continuous function from an $n$-sphere into Euclidean $n$-space maps some pair of antipodal points to the same point.
\end{theorem}

Another result which is useful to us is the Lyusternik–Schnirelmann theorem, logically equivalent to the Borsuk--Ulam theorem:

\begin{theorem}[Lyusternik–Schnirelmann~\cite{lusternik1930topological}]
\label{thrm:LS}
Let $n \ge 0$, and let $\{U_1,\ldots, U_{n+1}\}$ be a closed cover of $S^{n}$.
Then there exists an index $i \in \{1,\ldots,n+1\}$ such that $U_i$ contains a pair of antipodal points.
\end{theorem}

\begin{definition}[Odd function]
\label{def:oddmap}
Let $X\subseteq \R^n$ and $Y \subseteq \R^m$ (or $\ell^2$) be symmetric subsets, meaning $-x\in X$ whenever $x\in X$ and $-y\in Y$ whenever $y\in Y$.
A function $f\colon X\rightarrow Y$ is \emph{odd} if for every $x \in X$, we have $f(-x)=-f(x)$.
\end{definition}

To apply the Borsuk--Ulam theorem, we will use triangulations that respect the antipodal symmetry of spheres.

\begin{definition}[Antipode-preserving triangulation]
A triangulation $\tau$ of $S^{k}$ is \emph{antipode-preserving} if for every simplex $\sigma \in \tau$, the simplex $-\sigma \coloneqq \{ -v : v \in \sigma \}$ is also in $\tau$.
\end{definition}

To form an antipode-preserving triangulation of $S^2$, first triangulate its equator $S^1$ in an antipodal way, then extend this to a triangulation of the northern hemisphere arbitrarily, and then reflect this triangulation through the origin to the southern hemisphere.
Figures~\ref{fig1} and \ref{fig2} depict antipode-preserving $\varepsilon$-triangulations of $S^1$ and $S^2$ respectively.

\begin{figure}
\centering
\begin{minipage}{.5\textwidth}
\centering
\includegraphics[width=.3\linewidth]{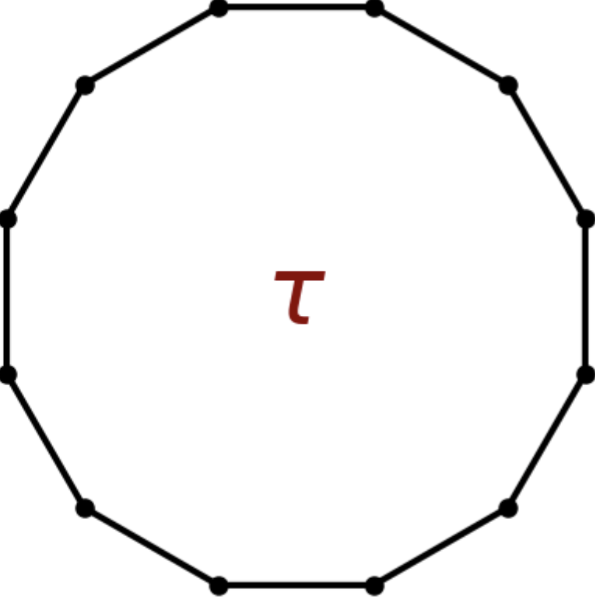}
\captionof{figure}{\small{An antipode-preserving triangulation $\tau$ of $S^1$.}}
\label{fig1}
\end{minipage}%
\begin{minipage}{.5\textwidth}
\centering
\includegraphics[width=.4\linewidth]{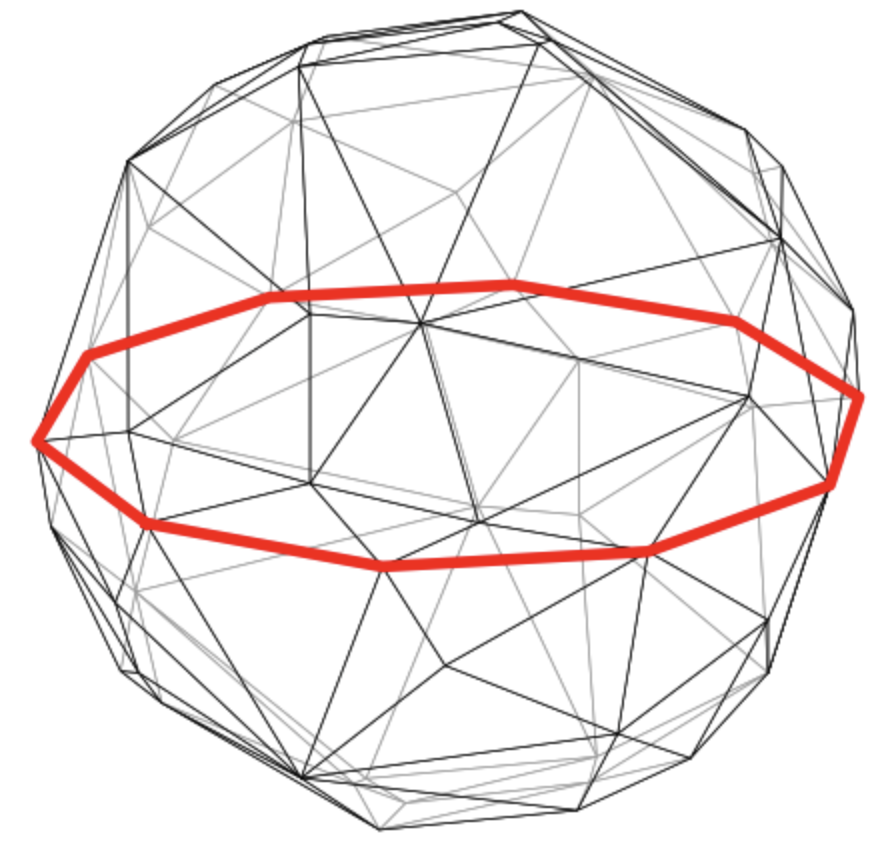}
\captionof{figure}{An antipode-preserving triangulation $\tau$ of $S^2$.
}
\label{fig2}
\end{minipage}
\end{figure}

An arbitrary point in $|\tau|$, the geometric realization of the triangulation $\tau$, can be written as a formal sum $\sum_{i=0}^k\lambda_i v_i$, with $v_i \in V(\tau)$,
where the barycentric coordinates $\{\lambda_i\}$ satisfy $0 \le \lambda_i \le 1$ and $\sum_{i} \lambda_{i}=1$.
This representation is crucial for constructing continuous extensions of maps in Section~\ref{secBorsukUlam}.

\section{`Invariance of Dimension' approach to lower bound $d_\gh(B^m,B^n)$}
\label{sec:invOFdim}

Our initial strategy to establish a non-trivial lower bound for $d_\gh(B^m,B^n)$ adapts the classical proof of `Invariance of Dimension'~\cite{munkres2000topology}.
Recall one can prove that Euclidean space $\R^m$ is not homeomorphic to $\R^n$ for $m>n$, as follows.
Suppose for a contradiction we had a homeomorphism $h \colon \R^m \rightarrow \R^n$.
Then restriction yields a homeomorphism $h \colon \R^m \setminus \{\vec{0}\} \rightarrow \R^n \setminus \{h(\vec{0})\}$.
However, we have homology groups $H_{m-1}(\R^m \setminus \{\vec{0}\})\cong \Z$ whereas $H_{m-1}(\R^n \setminus \{h(\vec{0})\})=0$ for $m > n$, so no such homeomorphism can exist.

The following proposition translates this structural obstruction into a quantitative metric bound by using the machinery of ambient \v{C}ech complexes. 
Although the resulting bound is superseded by better bounds in Section~\ref{secBorsukUlam}, we include it here because the `Invariance of Dimension' argument served as a conceptual motivation for our subsequent developments.

\constantlowerbound*

\begin{proof}
Suppose for a contradiction that $2d_\gh(B^m,B^n)<r$ for $r<\frac{1}{4}$.
By~\eqref{eq:def-gh-dis-codis}, where
\[2d_\gh(B^m,B^n)=\inf_{g,h}\max \{\dis(g),\dis(h), \codis(g,h)\},\]
there exist functions 
$h\colon B^m \rightarrow B^n$ and
$g\colon B^n \rightarrow B^m$ satisfying $\dis(g),\dis(h), \codis(g,h)< r$.
Since $r<\frac{1}{4}$, we may choose $\varepsilon>0$ and $c<1$ such that $4r+2\varepsilon < c < 1$.

Consider the scaled concentric balls $cB^m\coloneqq \{cx \mid x\in B^m\}$ and $(c-r)B^m$, which are strictly smaller than the unit ball.
We have the natural inclusion:
\[B^m \setminus cB^m \hookrightarrow B^m \setminus (c-r)B^m.\] 
Each of these punctured spaces above is homotopy equivalent to $S^{m-1}$, and the inclusion map is a homotopy equivalence.
Removing smaller interior balls is analogous to removing a point in the classical proof of the `Invariance of Dimension'.

We now apply Lemma~\ref{lem:ambientCech} with $Z=B^m$, $X=B^n$, and $r > 2d_\gh(X,Z)$
with the standard inclusion $B^n \hookrightarrow B^m$ given by $(x_1,x_2, \ldots, x_n)\mapsto (x_1,x_2, \ldots, x_n,0,\ldots,0)$.
As proven in Lemma~\ref{lem:ambientCech}, the function $h\colon B^m \rightarrow B^n$ extends to a simplicial map $\bar{h}\colon \cech{B^m}{\varepsilon}\rightarrow \cech{B^n}{r+\varepsilon}$ on ambient \v{C}ech complexes defined by $\bar{h}([z_0,\ldots,z_k])=[h(z_0),\ldots,h(z_k)]$.
Similarly, the function $g\colon B^n \rightarrow B^m$ extends to a simplicial map $\bar{g}\colon \cech{B^n}{r+\varepsilon}\rightarrow \cech{B^m}{3r+2\varepsilon}$ defined analogously.
Furthermore the composition $\bar{g}\circ \bar{h}$
\[\cech{B^m}{\varepsilon}\xrightarrow{\bar{h}} \cech{B^n}{r+\varepsilon}\xrightarrow{\bar{g}} \cech{B^m}{3r+2\varepsilon}\]
is contiguous to the inclusion $i: \cech{B^m}{\varepsilon}\hookrightarrow\cech{B^m}{3r+2\varepsilon}$.
All \v{C}ech complexes are ambient \v{C}ech complexes in $B^m$.

Note that $h(B^m\setminus cB^m)\subseteq B^n \subseteq B^m$, where the last inclusion is obtained by appending $m-n$ zeroes to each $n$-dimensional vector.
Restriction gives
\[\cech{B^m\setminus cB^m}{\varepsilon}\xrightarrow{\bar{h}} \cech{h(B^m\setminus cB^m)}{r+\varepsilon}\xrightarrow{\bar{g}} \cech{B^m\setminus (c-r)B^m}{3r+2\varepsilon}.\]
To see that the composition $\bar{g}\circ \bar{h}$ lands in the \v{C}ech complex on $B^m \setminus (c-r)B^m$, note that $\codis(g,h)< r$ implies $d(x,g(h(x)))<r$ for all $x \in B^m$.
Thus, $x\in B^m \setminus cB^m$ implies $g(h(x)) \in B^m \setminus (c-r)B^m$.
Consequently, the composition is contiguous to, and hence homotopic to, the inclusion map.

Since $4r+ 2\varepsilon < c$, we have $3r+2\varepsilon < (c-r)$.
Hence $\cech{B^m\setminus (c-r)B^m}{3r+2\varepsilon}$ is homotopy equivalent to the sphere $S^{m-1}$ by the nerve lemma~\cite{Borsuk1948,Dieck,Hatcher}.
By similar reasoning, we have $\cech{B^m\setminus cB^m}{\varepsilon}\simeq S^{m-1}$.
Therefore, the composition $\bar{g}\circ \bar{h}$ is homotopic to the inclusion $\cech{B^m\setminus cB^m}{\varepsilon}\hookrightarrow \cech{B^m\setminus (c-r)B^m}{3r+2\varepsilon}$, which by functoriality of the nerve lemma represents a homotopy equivalence between spaces homotopy equivalent to $S^{m-1}$.

Recall $h(B^m\setminus cB^m)\subseteq B^n\subseteq B^m$.
Therefore, the ambient \v{C}ech complex $\cech{h(B^m\setminus cB^m)}{r+\varepsilon}$ with balls taken in $B^m$ equals the ambient \v{C}ech complex $\cech{h(B^m\setminus cB^m)}{r+\varepsilon}$ with balls taken in $B^n$.
Since $m>n$, the union of balls in $B^n$ centered at each point in $h(B^m\setminus cB^m)$ has trivial homology group $H_{m-1}$.
By the nerve lemma, the \v{C}ech complex $\cech{h(B^m\setminus cB^m)}{r+\varepsilon}$ has trivial $H_{m-1}$.
However, this contradicts the fact that $\bar{g}\circ \bar{h}$ is a homotopy equivalence between spaces homotopy equivalent to $S^{m-1}$, which has $H_{m-1}(S^{m-1})\cong \Z$.

Hence we must have $2d_\gh(B^m, B^n) \ge \frac{1}{4}$, which gives $d_\gh(B^m, B^n) \ge \frac{1}{8}$.
\end{proof}

The lower bound $d_\gh(B^m, B^n) \ge \frac{1}{8}$ in Proposition~\ref{prop:constant-lower-bound} does not depend on specific dimensions $m$ and $n$, nor on the difference between them.
See Theorem~\ref{thm:general} for a result depending on $n$, and Proposition~\ref{prop:covering} for a result depending on both $m$ and $n$.

\section{Using the Borsuk--Ulam theorem to lower bound $d_\gh(B^m,B^n)$}
\label{secBorsukUlam}

In this section, we use the Borsuk--Ulam theorem to establish a lower bound on the Gromov--Hausdorff distance between unit balls of different dimensions.
The core strategy is to show that any function $B^m \rightarrow B^n$ for $m>n$ must exhibit a certain minimum distortion.
For our bound, it suffices to focus on the restriction of this function to the boundary sphere $S^{m-1} \rightarrow B^n$.
Our method proceeds by constructing a continuous, piecewise linear approximation of this restricted function on an antipode-preserving $\varepsilon$-triangulation of $S^{m-1}$ (see Definition~\ref{def:ep-triang}).
Applying the Borsuk--Ulam theorem to this continuous approximation produces two antipodal simplices in the $\varepsilon$-triangulation whose images overlap, providing a quantifiable lower bound on the distortion of the function $S^{m-1}\to B^n$.
By~\eqref{eq:def-gh-dis-codis}, this yields a lower bound for the Gromov--Hausdorff distance $d_\gh(B^m,B^n)$.

\begin{definition}[A geometric quantity for intersecting convex hulls]
\label{def:alphan}
For each $n\geq 1$, define 
\[\alpha_n=\sup \left\{d(X,Y) \mid X,Y\subseteq \R^n \text{ are finite},\ \diam(X),\diam(Y)\leq 1,\, \conv(X)\cap\conv(Y) \neq \emptyset\right\}.\]
\end{definition}

Roughly speaking, $\alpha_n$ upper bounds the minimum distance between the vertex sets of two intersecting simplices of diameter at most $1$.

In Lemma~\ref{lem:algebraic-general} we find the exact values of $\alpha_n$ for all $n\geq 1$.
The proof of this lemma appears in Appendix~\ref{app:proofoflemmageneral}.

\begin{lemma}
\label{lem:algebraic-general}
We have $\alpha_n=
\begin{cases}
 \sqrt{\frac{n}{n+2}},&n\text{ even},\\
 \sqrt{\frac{n^2+2n-1}{(n+1)(n+3)}},&n\text{ odd}.
\end{cases}$
\end{lemma}

Observe that for large $n$, $ \sqrt{\frac{n}{n+2}}\rightarrow 1$ and $\sqrt{\frac{n^2+2n-1}{(n+1)(n+3)}}\rightarrow 1$.
Furthermore, $\alpha_n$ strictly increases to $1$ as $n\rightarrow\infty$, while $\alpha_n<1$ for every finite $n$.

\subsection{A Borsuk--Ulam bound}

Using Lemma~\ref{lem:algebraic-general}, we can now prove the following theorem.

\maintheorem*

This lower bound strictly decreases to $\frac{1}{2}$ as $n\to \infty$.

\begin{proof}
Let $m>n$, i.e., let $m-1\ge n$.
By~\eqref{eq:def-gh-dis-codis}, we have $d_\gh(B^m,B^n)\geq \inf\{\frac{\dis(F)}{2}\}$, where the infimum is taken over all (possibly discontinuous) functions $F\colon B^m \to B^n$.
So, consider an arbitrary function $F\colon B^m \rightarrow B^n$.
To control the distortion of $F$, we will in fact control the distortion of its restriction $f=F|_{S^{m-1}}$ to the boundary sphere $S^{m-1}$.
Note $\dis(F) \ge \dis(f)$.

Let $\varepsilon>0$.
Consider an antipode-preserving $\varepsilon$-triangulation $\tau$ of $S^{m-1}$ with vertex set $V(\tau)$.
So $\tau\cong S^{m-1}$.
From the possibly discontinuous function $f\colon S^{m-1}\to B^n$ we define a \emph{continuous} function $\bar{f}\colon \tau \rightarrow B^n$, as follows.
First, for each $v\in V(\tau)$, let $\bar{f}(v)=f(v)$.
Second, on a point in the geometric realization of $\tau$, define $\bar{f}(\sum_{i=0}^k\lambda_i v_i)=\sum_{i=0}^k\lambda_i f(v_i)\in B^n$, where this last sum is a convex combination of points in the convex set $B^n$.

By the Borsuk--Ulam theorem, there are two antipodal points $x,-x \in \tau$ that are mapped to the same location by $\bar{f}$, i.e., $\bar{f}(-x)=\bar{f}(x)$.
Let $[u_0, \ldots,u_{m-1}]$ and $[-u_0, \ldots,-u_{m-1}]$ be antipodal $(m-1)$-dimensional simplices in $\tau$ which contain $x$ and its antipode $-x$, respectively.

For any $0\le i,j\le m-1$, we have $d(\bar{f}(u_i),\bar{f}(u_j))\le \dis(f)+d(u_i,u_j)\le \dis(f)+\varepsilon$.
Consider the finite subsets $X=\{f(u_0),\ldots,f(u_{m-1})\}$ and $Y=\{f(-u_0),\ldots,f(-u_{m-1})\}$ of $B^n$.
The image of the simplex $[u_0, \ldots,u_{m-1}]$ under $\bar{f}$ is $\bar{f}([u_0, \ldots,u_{m-1}])=\conv(X)$, and similarly $\bar{f}([-u_0, \ldots,-u_{m-1}])=\conv(Y)$.
Both $\conv(X)$ and $\conv(Y)$ are compact sets in $B^n$, each having diameter at most $\dis(f)+\varepsilon$.

The condition $\bar{f}(-x)=\bar{f}(x)$ implies that $\conv(X)\cap \conv(Y)\neq \emptyset$.
Let $u^*\in \{u_0, \ldots,u_{m-1}\}$ and $v^*\in \{-u_0, \ldots,-u_{m-1}\}$ be the vertices that realize the minimum distance between the sets $X$ and $Y$.
By Definition~\ref{def:alphan},  this distance is bounded above by $(\dis(f)+\varepsilon) \alpha_n$, giving $d(\bar{f}(u^*),\bar{f}(v^*)) \leq (\dis(f)+\varepsilon) \alpha_n$.

Using~\eqref{eq:dis_calc}, we have
\[
\dis(f) \geq d(u^*,v^*) - d(\bar{f}(u^*),\bar{f}(v^*)) \geq (2-2\varepsilon) - (\dis(f)+\varepsilon) \alpha_n.
\]
So $\dis(f)(1+\alpha_n) \geq 2-\varepsilon\left(2+\alpha_n\right)$, which gives
\[\dis(f) \geq \tfrac{2}{\left(1+\alpha_n\right)}-\varepsilon\, \tfrac{2+\alpha_n}{\left(1+\alpha_n\right)}.\]

In summary, any function $F\colon B^m\to B^n$ satisfies $\frac{\dis(F)}{2} \geq \frac{\dis(f)}{2}\geq \tfrac{1}{1+\alpha_n} - \tfrac{\varepsilon}{2}\,\tfrac{2+\alpha_n}{\left(1+\alpha_n\right)}$ for any $\varepsilon>0$.
Hence
$d_\gh(B^m,B^n)\geq \inf\{\frac{\dis(F)}{2}\}\geq \tfrac{1}{1+\alpha_n}$.
Substituting the value of $\alpha_n$ from Lemma~\ref{lem:algebraic-general}, we obtain the required result.
\end{proof}

\begin{remark}
The same proof gives $d_\gh(S^{m-1},B^n) \ge \tfrac{1}{1+\alpha_n}$ for $n\ge 1$ and $m>n$.
\end{remark}

Indeed, the argument used in the proof of Theorem~\ref{thm:general} immediately restricts any function $F\colon B^m \to B^n$ to $f=F|_{S^{m-1}}$, and all subsequent steps (the $\varepsilon$-triangulation, Borsuk--Ulam argument, and distortion bound) involve only $f\colon S^{m-1} \to B^n$.
The same argument therefore applies to any function $S^{m-1} \to B^n$ directly, which by~\eqref{eq:def-gh-dis-codis} gives $d_\gh(S^{m-1},B^n) \ge \tfrac{1}{1+\alpha_n}$.

Figure~\ref{final} illustrates the proof of $d_\gh(B^2,B^1) \ge \tfrac{1}{1+\alpha_1} = \frac{2}{3}$.

\begin{figure}[htb]
\centering
\includegraphics[width=0.7\textwidth]{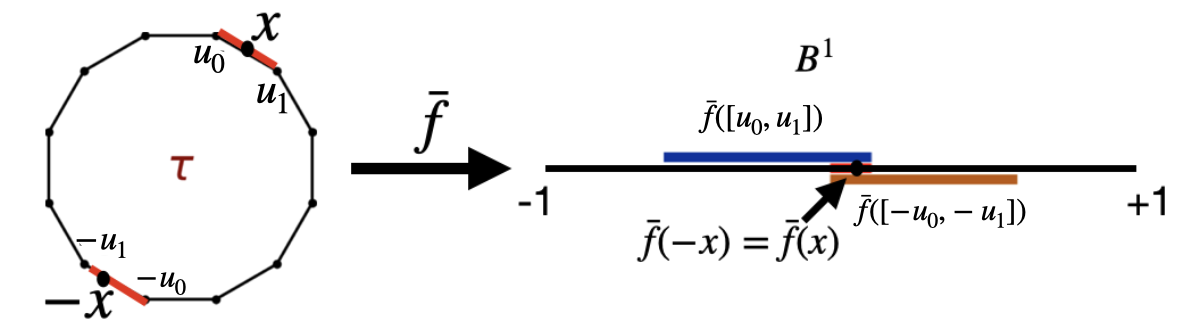}
\caption{
Geometric visualization of a map $\bar{f}\colon S^{m-1}\rightarrow B^n$ with $m=2$ and $n=1$.
The images of the antipodal simplices $[u_0, u_1]$ and $[-u_0, -u_1]$ overlap since $\bar{f}(-x)=\bar{f}(x)$.
See~\cite[Lemma~2.3]{katz2020torus} for a related construction.
}
\label{final}
\end{figure}

\subsection{An asymptotic result using covering radii}

The techniques developed in the following propositions and lemmas are closely adapted from the framework established by Lim, Mémoli and Smith~\cite{lim2023gromov}, for bounding the Gromov--Hausdorff distance between spheres.
We derive the exact Gromov--Hausdorff distance $d_\gh(B^m,B^n)$ when $n=0$ or when $m=\infty$.
Following also~\cite[Lemma~5.10]
{colding1996large} and~\cite{funano2008estimates}, we establish a lower bound using the covering radius of balls in Proposition~\ref{prop:covering}.
In one step, we use the Lyusternik–Schnirelmann theorem, which is logically equivalent to the Borsuk--Ulam theorem~\cite{matousek2003using}.

The first step is to compare a unit ball to a finite metric space $P$.
Lemma~\ref{lemma:P} and Proposition~\ref{prop:equalto1} provide the groundwork we need to prove the covering radius result in Proposition~\ref{prop:covering}.

\begin{lemma}
\label{lemma:P}
Let $m\ge 1$ and let $P$ be a finite metric space with $|P| \le m$.
Then $d_\gh(B^m,P)\geq 1$, and $d_\gh(B^m,P)=1$ if $\diam (P)\leq 2$.
\end{lemma}

\begin{proof}
Let $C$ be a correspondence between $B^m$ and $P$.
Restrict $C$ to obtain a correspondence between $S^{m-1}$ and $P'=\{p\in P~:~(x,p)\in C\text{ for some }x\in S^{m-1}\}$.
For $p\in P'$, let $C(p)=\{z\in S^{m-1}\mid (z,p)\in C\}$.
Since $C$ is a correspondence, $\{C(p)\}_{p\in P'}$ is a cover of $S^{m-1}$, and $\{\overline{C(p)}\}_{p\in P'}$ is a closed cover of $S^{m-1}$.
Since $\lvert P'\rvert\leq m$, Lyusternik–Schnirelmann (Theorem~\ref{thrm:LS}) gives $\diam(\overline{C(p_0)})\geq 2$ for some $p_0\in P$.
Hence $\dis(C) \ge \diam(C(p_0)) = \diam(\overline{C(p_0)}) \geq 2$, giving $d_\gh(B^m,P)\geq 1$.
If $\diam (P)\leq 2$, then~\eqref{eq:lessthan1forall} lets us conclude $d_\gh(B^m,P)=1$.
\end{proof}

\covering*
This bound is in terms of both $m$ and $n$.

\begin{proof}
Let $P$ be any nonempty subset of $B^n$ with $|P| \le m$.
The triangle inequality and Lemma~\ref{lemma:P} give
\[d_\gh(B^m,B^n)
\ge d_\gh(B^m,P) - d_\gh(P,B^n)
\ge d_\gh(B^m,P) - d_\h(P,B^n)
= 1 - d_\h(P,B^n).\]
By the definition of covering radius~\eqref{eq:defcoveringradius}, we obtain the claim by taking the infimum over all such possible choices of $P$.
\end{proof}


Figure~\ref{fig:Prop4vsThrm2} compares the lower bounds for $d_\gh(B^m, B^n)$ obtained via Proposition~\ref{prop:covering} and Theorem~\ref{thm:general}.
For $n=1$, Theorem~\ref{thm:general} yields a constant lower bound of $d_\gh(B^m,B^1)\geq\tfrac{2}{3}$, which is strictly better than the lower bounds from Proposition~\ref{prop:covering} for $m$ up to 2.
Proposition~\ref{prop:covering} matches it at $m=3$ and surpasses it for all $m\geq4$.

For $n=2$, Theorem~\ref{thm:general} establishes a lower bound of $d_\gh(B^m,B^2)\geq\tfrac{\sqrt{2}}{\sqrt{2}+1}\approx 0.586$, providing a better lower bound for $m$ up to 8, while Proposition~\ref{prop:covering} approximately matches it at $m=9$ and surpasses it for $m\geq10$.
In this case, the explicit covering radii for $B^2$ utilized in Proposition~\ref{prop:covering} are taken from the optimal configurations compiled by Friedman~\cite{friedman2021circles}.

For $n=3$, Theorem~\ref{thm:general} provides a lower bound of $d_\gh(B^m,B^3)\geq\tfrac{2\sqrt{3}}{2\sqrt{3}+\sqrt{7}}\approx 0.567$.
This remains strictly better than the lower bounds from Proposition~\ref{prop:covering} for $m$ up to at least 24~\cite{bezdek2015covering,wynn2012mathoverflow}.

Similarly, for $n=4$ and $5$, Theorem~\ref{thm:general} gives better lower bounds for $d_\gh(B^m, B^n)$ for $m$ up to at least 48 and 113, respectively~\cite{glazyrin2019covering, verger2005covering, li2011concise}.

For higher values of $n$ ($n\geq 3$), the exact values of $\cov_{B^n}(m)$ remain unknown for most $m$, necessitating the use of upper and lower bounds for $\cov_{B^n}(m)$.
When restricting attention to the Gromov--Hausdorff distance between balls of consecutive dimensions, $d_\gh(B^{n+1},B^n)$, Theorem~\ref{thm:general} always gives better lower bounds compared to those coming from Proposition~\ref{prop:covering}~\cite{boroczky2004finite}.

We conclude this section by considering $B^\infty$.

\begin{proposition}
\label{prop:equalto1}
We have $d_\gh(B^\infty,B^n)=1$ for any integer $n\ge 0$.
\end{proposition}

\begin{proof}
Let $P$ be a finite metric space. As in the proof of Lemma~\ref{lemma:P}, a correspondence $C$ between $B^\infty$ and $P$ induces a closed cover of $S^\infty$, the boundary of $B^\infty$.
Thus, it induces a closed cover of any finite dimensional sphere $S^{\lvert P \rvert -1}\subseteq S^\infty$.
Again using Theorem~\ref{thrm:LS}, we get $\dis(C)\geq 2$ and $d_\gh(B^\infty,P)\geq 1$.

Fix $\varepsilon>0$ and let $P_\varepsilon\subseteq B^n$ be a finite $\varepsilon$-net for $B^n$.
By the triangle inequality, we have
\[d_\gh(B^\infty,B^n)
\geq d_\gh(B^\infty,P_\varepsilon)-d_\gh(B^n,P_\varepsilon)
\geq d_\gh(B^\infty,P_\varepsilon)-d_\h(B^n,P_\varepsilon)
\geq 1-\varepsilon.\]
Since $\varepsilon >0$ was arbitrary, we have $d_\gh(B^\infty,B^n)\geq 1$.
And~\eqref{eq:lessthan1forall} gives $d_\gh(B^\infty,B^n)=1$.
\end{proof}

The same proof gives $d_\gh(B^\infty,X)\ge 1$ for any totally bounded metric space $X$.

\section{Upper bounds on $d_\gh(B^m,B^n)$}
\label{sec:upperbounds}

We explore whether the techniques used to establish the upper bounds on the Gromov--Hausdorff distance between spheres can be extended to prove better upper bounds on $d_\gh(B^m,B^n)$ for $m>n\geq 1$.
In~\cite[Theorem~E]{lim2023gromov}, Lim, Mémoli, and Smith use space-filling curves to prove there exists a continuous odd surjection $\psi_{m,n}$ from the low-dimensional sphere $S^n$ to the high-dimensional sphere $S^m$.
The graph of this surjection constitutes a correspondence whose distortion proves $d_\gh(S^m,S^n)<d_\h(S^m,S^n)$.

We use the framework from~\cite{lim2023gromov}, namely continuous odd surjections, to show that the distance between unit balls of distinct dimensions is strictly less than $1$.
Recall from Definition~\ref{def:oddmap} that a function $f\colon S^n\to S^m$ is \emph{odd} if 
$f(-x)=-f(x)$ for every $x\in S^n$.
We prove:

\upperboundlessthanone*

In order to prove Theorem~\ref{thrm:upperbound}, we begin by establishing the existence of a continuous odd surjection from a lower-dimensional ball to a higher-dimensional one.
Please refer to Appendix~\ref{app:oddsctssurjective} for the proof.

\begin{theorem}
\label{thrm:lowtohigh}
For $m>n\geq 1$, there exists a continuous odd surjection $\psi_{m,n}\colon B^n\rightarrow B^m$ such that $\psi_{m,n}(S^{n-1})\cap \{\vec{0}\}=\emptyset$.
\end{theorem}

With this, the proof of Theorem~\ref{thrm:upperbound} now follows:

\begin{proof}[Proof of Theorem~\ref{thrm:upperbound}]
Let $\psi_{m,n}\colon B^n\rightarrow B^m$ be the map given in Theorem~\ref{thrm:lowtohigh}.
Recall that the graph of a surjection is a correspondence and let $C_{m,n}\coloneqq \graph(\psi_{m,n})$.
It suffices to show $\dis(C_{m,n})=\dis(\psi_{m,n})<2$.

Since $\psi_{m,n}$ is continuous and $B^n$ is compact, the supremum in the definition of distortion is attained,
\[\dis(\psi_{m,n})=\sup\limits_{x,x'\in B^n}\lvert d_{B^n}(x,x')-d_{B^m}(\psi_{m,n}(x),\psi_{m,n}(x'))\rvert.\]
Let $x_0,x'_0\in B^n$ attain the supremum above.
It must be that $x_0\neq x'_0$, since $x_0=x_0'$ would imply $d_\gh(B^m,B^n)\leq \frac{1}{2}\dis(C_{m,n})=\frac{1}{2}\dis(\psi_{m,n})=0$, i.e.\ $d_\gh(B^m,B^n)=0$ so $B^m$ and $B^n$ are isometric, which is a contradiction (for example by Theorem~\ref{thm:general}) since $m\neq n$.
There are two cases.

\emph{Case 1: $x'_0\neq -x_0$:} 
In this case,
$0<d_{B^n}(x_0,x'_0)<2$ 
and 
$0\leq d_{B^m}(\psi_{m,n}(x_0),\psi_{m,n}(x'_0))\leq 2$.
Thus, $\lvert d_{B^n}(x_0,x'_0)-d_{B^m}(\psi_{m,n}(x_0),\psi_{m,n}(x'_0))\rvert <2$.

\emph{Case 2: $x'_0= -x_0$:} 
Since $x_0\neq x_0'$, we have $x_0\neq 0$.
In this case, we can write $d_{B^n}(x_0,x'_0)= 2\lVert x_0\rVert_{B^n}$ and $d_{B^m}(\psi_{m,n}(x_0),\psi_{m,n}(x'_0))=2\lVert\psi_{m,n}(x_0)\rVert_{B^m}$.
Therefore, $\dis(\psi_{m,n})=2\lvert \lVert x_0\rVert_{B^n}-\lVert\psi_{m,n}(x_0)\rVert_{B^m}\rvert$.
If $0<\lVert x_0\rVert_{B^n}<1$, then $0\leq \lVert\psi_{m,n}(x_0)\rVert_{B^m}\leq 1$ gives $\dis(\psi_{m,n})<2$, and if $\lVert x_0\rVert_{B^n}=1$, then $0< \lVert\psi_{m,n}(x_0)\rVert_{B^m}\leq 1$ gives $\dis(\psi_{m,n})<2$.
\end{proof}

\section{Conclusion and open questions}
\label{sec:conclusion}

We have investigated the question of determining the Gromov--Hausdorff distance $d_\gh(B^m,B^n)$ between Euclidean unit balls of different dimensions, a problem where the stability of persistent homology provides no positive lower bounds.
To address this, we use classic tools from algebraic topology.
Our first approach in Section~\ref{sec:invOFdim} adapts the classical `Invariance of Dimension' proof to yield the universal lower bound $d_\gh(B^m,B^n) \geq \frac{1}{8}$ for all $m>n\geq 1$.
Our approach in Section~\ref{secBorsukUlam} employs the Borsuk--Ulam theorem via antipode-preserving triangulations to obtain a dimension-dependent lower bound $d_\gh(B^m,B^n) \geq \tfrac{1}{1+\alpha_n}> \frac{1}{2}$.
We establish the bound $d_\gh(B^m,B^n)\geq 1-\cov_{B^n}(m)$, which dictates that as $m\rightarrow \infty$, the covering radius vanishes and the distance approaches $1$.
Additionally, we show that the Gromov--Hausdorff distance between Euclidean unit balls of different dimensions is always strictly less than $1$.

We end with some open questions.

\begin{question}
For a fixed dimension $n\geq 1$, is $d_\gh(B^m,B^n)$ a nondecreasing function of $m>n$?

See also~\cite[Question~I]{lim2023gromov}.
\end{question}

\begin{question}
Our lower bound in Theorem~\ref{thm:general} surpasses the covering radius bound from Proposition~\ref{prop:covering} for some $m$ and $n$ (see Figure~\ref{fig:Prop4vsThrm2}).
But it depends only on $n$, and not also on $m$.
Hence it does not capture the intuition that the geometric dissimilarity should increase with the dimensional difference between the spaces.
Can one establish an improved version of Theorem~\ref{thm:general} that is a function of both $m$ and $n$?
\end{question}

\begin{question}
How do $d_\gh(B^m,B^n)$ and $d_\gh(S^{m-1},S^{n-1})$ relate for $m>n$, where the spheres are equipped with the restriction of the Euclidean metric rather than the geodesic metric (see~\cite{lim2023gromov,GH-BU-VR,harrison2023quantitative,rodriguez2026some})? 
For a map $f\colon S^{m-1} \rightarrow S^{n-1}$, consider the radial extension $\bar{f}\colon B^m \rightarrow B^n$ defined as $\bar{f}(p)= \|p\| f(\frac{p}{\|p\|})$ for $p\neq 0$ and $\bar{f}(0)=0$; do connections between $\dis(f)$ and $\dis(\bar{f})$ yield a relationship between these Gromov--Hausdorff distances?

\end{question}

\begin{question}

Can the methods developed in this paper be extended to bound $d_\gh(B^m_p, B^n_p)$ for $\ell^p$ unit balls 
\[B^n_p = \left\{x \in \R^n : \left(\sum_{i=1}^n |x_i|^p\right)^{1/p} \leq 1\right\}\]
for $1\le p \le \infty$?
Or to ellipsoids $\{x \in \R^n : x^T A x \leq 1\}$ for positive definite matrices $A \neq I$? 
Since these sets are centrally symmetric ($x \in B$ implies $-x \in B$), what do Borsuk--Ulam approaches yield in this setting?
\end{question}

\begin{question}
More generally, how do we determine the Gromov--Hausdorff distance between polynomially defined convex bodies of different dimensions or with different defining parameters?
\end{question}

\bibliographystyle{plain}
\bibliography{GHdistanceBetweenBalls-Final/GHdistancesBalls}

\newpage

\appendix

\section{The exact value of $\alpha_n$}
\label{app:proofoflemmageneral}

This appendix gives the proof of Lemma~\ref{lem:algebraic-general}, which states the following.
For $n\ge 1$, recall $\alpha_n=\sup \{d(X,Y) \mid X,Y\subseteq \R^n \text{ are finite},\ \diam(X),\diam(Y)\leq 1,\, \conv(X)\cap\conv(Y) \neq \emptyset\}$.
Let $p=\lfloor\frac n2\rfloor$ and $q=\lceil\frac n2\rceil$.
Then, 
$\alpha_n^2=\frac12\left(\frac{p}{p+1}+\frac{q}{q+1}\right)$.
Equivalently,
\[\alpha_n=
\begin{cases}
 \sqrt{\frac{n}{n+2}},&n\text{ even},\\
 \sqrt{\frac{n^2+2n-1}{(n+1)(n+3)}},&n\text{ odd}.
\end{cases}\]

\begin{proof}[Proof of Lemma~\ref{lem:algebraic-general}]
Let $X,Y\subseteq\R^n$ be finite sets as above and let $z$ be a vertex of the nonempty convex intersection, $K=\conv(X)\cap\conv(Y)$.
Choose subsets $X_0=\{x_0,\ldots,x_r\}\subseteq X$ and $Y_0=\{y_0,\ldots,y_s\}\subseteq Y$ of minimal cardinality such that $z\in\conv(X_0)\cap\conv(Y_0)$.
By minimality, $X_0$ and $Y_0$ are affinely independent, and all
coefficients in the convex representation of $z$ using these points
are nonzero.
Thus,
\[z=\sum_{i=0}^r a_i x_i
=\sum_{j=0}^s b_j y_j
\quad\text{with}\quad 
a_i,b_j>0
\quad\text{and}\quad 
\sum_{i=0}^r a_i=\sum_{j=0}^s b_j=1.\]
Set $U=\operatorname{span}\{x_i-z:0\leq i\leq r\}$ and $V=\operatorname{span}\{y_j-z:0\leq j\leq s\}$.
We claim that $U\cap V=\{0\}$.
Suppose that $0\neq w\in U\cap V$. 
Since $z$ lies in the relative interior of $\conv(X_0)$ and $\conv(Y_0)$, for sufficiently small $\varepsilon>0$, one has $z\pm\varepsilon w\in\conv(X_0)\cap\conv(Y_0)\subseteq K$.
This contradicts the fact that $z$ is vertex of $K$.
Therefore $r+s=\dim U+\dim V=\dim(U+V)\leq n$.

Since $\|x_i-y_j\|\geq d(X,Y)$, we have
\begin{align*}
d(X,Y)^2 &\leq \sum_{i=0}^r\sum_{j=0}^s a_i b_j\|x_i -y_j\|^2 \\
&= \sum_{i,j}a_i b_j\|(x_i-z)-(y_j-z)\|^2 \nonumber \\
&=\sum_{i,j}a_i b_j \Big[\|x_i-z\|^2+ \|y_j-z\|^2-2\langle x_i-z,y_j-z \rangle\Big] \nonumber \\
&= \sum_{i,j}a_i b_j \|x_i-z\|^2+\sum_{i,j}a_i b_j \|y_j-z\|^2-2\sum_{i,j}a_i b_j\langle x_i-z,y_j-z \rangle \nonumber \\
&= \sum_{i}a_i \|x_i-z\|^2+\sum_{j} b_j \|y_j-z\|^2-2\left\langle \sum_{i}a_i (x_i-z), \sum_{j} b_j (y_j-z)\right\rangle \nonumber \\
&= \sum_{i}a_i\|x_i-z\|^2+\sum_{j} b_j \|y_j-z\|^2
\end{align*}
where the last step follows since $\sum_{i} a_i x_i=z=\sum_{j} b_j y_j$ implies 
$\sum_{i} a_i (x_i-z)=0=\sum_{j} b_j (y_j-z)$.

Expanding
$x_i-x_k=(x_i-z)-(x_k-z)$ gives
\begin{align*}
\sum_{i,k}a_i a_k\|x_i-x_k\|^2
&=\sum_{i,k}a_i a_k
  \bigl(\|x_i-z\|^2+\|x_k-z\|^2
        -2\langle x_i-z,x_k-z\rangle\bigr)=2\sum_i a_i\|x_i-z\|^2.
\end{align*}
Since $\diam(X_0)\leq 1$,
\[
\sum_i a_i\|x_i-z\|^2
= \frac12\sum_{i,k}a_i a_k\|x_i-x_k\|^2
\leq \frac12\sum_{i\neq k}a_i a_k
= \frac12\left(1-\sum_i a_i^2\right)
\leq \frac12\left(1-\frac1{r+1}\right)
= \frac{r}{2(r+1)}
\]
where the last inequality is obtained by applying Cauchy--Schwarz to the vectors $(a_0,\ldots,a_r),(1,\ldots,1)\in \R^{r+1}$.
Similarly, $\sum_j b_j\|y_j-z\|^2 \leq \frac{s}{2(s+1)}$.

Combining, we get
\[d(X,Y)^2 \leq \frac{r}{2(r+1)}+\frac{s}{2(s+1)}.\]
Since $r+s\leq n$, the right-hand side is maximized when
$r+s=n$ and $|r-s|\leq1$.
Thus
$d(X,Y)^2 \leq \frac12\left(\frac{p}{p+1}+\frac{q}{q+1} \right)$.
This proves $\alpha_n^2 \leq \frac12\left(\frac{p}{p+1}+\frac{q}{q+1}\right).$

To prove sharpness, choose an orthogonal decomposition
$\R^n=M\oplus N$ with $\dim M=p$ and $\dim N=q$.
Let $X=\{x_0,\ldots,x_{p}\}\subseteq M$ be the vertex set of a regular $p$-simplex of side
length $1$ centered at the origin, and let $Y=\{y_0,\ldots,y_{q}\}\subseteq N$ be the vertex set
of a regular $q$-simplex of side length $1$ centered at the origin.  
Then $0\in\conv(X)\cap\conv(Y)$.
These vertex sets satisfy
\[\|x_i\|^2=\frac{p}{2(p+1)}
\quad\text{and}\quad
\|y_i\|^2=\frac{q}{2(q+1)}
\quad\text{for all }i.\]
Indeed, by the regularity of the simplex and its centered arrangement, the value of $\|x_i\|^2$ is constant for all $i$, and the value of $\langle x_i,x_j\rangle$ is constant for all $i\neq j$.
For $i\neq j$,
\begin{equation}
\label{eq:1}
1=\|x_i-x_j\|^2\\
=\|x_i\|^2+ \|x_j\|^2-2\langle x_i,x_j\rangle\\
= 2\|x_i\|^2-2\langle x_i,x_j\rangle
\end{equation}
Taking the inner product of the sum $\sum\limits_{i=0}^{p}x_i=\vec{0}$ with a specific vertex $x_j$ yields:
\begin{equation}
\label{eq:2}
0=\left\langle \sum\limits_{i=0}^{p}x_i, x_j\right\rangle\\
=\sum\limits_{i=0}^{p}\langle x_i, x_j\rangle\\
=\langle x_j, x_j\rangle + \sum\limits_{i\neq j}\langle x_i, x_j\rangle \\
= \|x_j\|^2 + p\langle x_i, x_j\rangle.
\end{equation}
Combining~\eqref{eq:1} and~\eqref{eq:2} gives $\|x_j\|^2=\frac{p}{2(p+1)}$ for every $j$.

Since $M$ and $N$ are orthogonal,
$\|x-y\|^2 = \|x\|^2+\|y\|^2 = \frac12\left(\frac{p}{p+1}+\frac{q}{q+1}\right)$
for every $(x,y)\in X\times Y$.  
Hence,  $\alpha_n^2 = \frac12\left(\frac{p}{p+1}+\frac{q}{q+1}\right).$
If $n$ is even, then $p=q=\frac{n}{2}$, giving
$\alpha_n^2=\frac{n}{n+2}.$
If $n$ is odd, then $p=\frac{n-1}{2}$ and $q=\frac{n+1}{2}$, giving
$\alpha_n^2= \frac{n^2+2n-1}{(n+1)(n+3)}.$
\end{proof}

\section{Proof of Theorem~\ref{thrm:lowtohigh}: Odd continuous surjections using space-filling curves}
\label{app:oddsctssurjective}

We use the following well-known result on space-filling curves by Peano~\cite{peano1890courbe}.

\begin{theorem}
\label{thrm:spacefilling}
There exists a continuous surjection $H\colon [0,1] \to [0,1]^2.$
\end{theorem}

We now prove Theorem~\ref{thrm:lowtohigh}, which states that for $m>n\geq 1$, there exists a continuous odd surjection $\psi_{m,n}\colon B^n\rightarrow B^m$ such that $\psi_{m,n}(S^{n-1})\cap \{\vec{0}\}=\emptyset$.

\begin{proof}[Proof of Theorem~\ref{thrm:lowtohigh}]

We divide this proof into four parts.

\textbf{(i) Construction of a continuous odd surjection $s_2\colon [-1,1]\rightarrow [-1,1]^2$.} 
Let $H\colon [0,1]\to[0,1]^2$ be a variant of the continuous surjection from Theorem~\ref{thrm:spacefilling} parameterized so that the initial point maps to the midpoint of the lower boundary: $H(0)=(\frac{1}{2},0)$.
Define a continuous bijection $T\colon[0,1]^2\rightarrow [-1,1]\times [0,1]$ by $T(u,v)=(2u-1,v)$.
Note $T(\frac{1}{2},0)=(0,0)$.

Define $\phi_{\half}\colon [0,1]\rightarrow[-1,1]\times [0,1]$ by $\phi_{\half}(t)=(T\circ H)(t)$.
As the composition of two continuous surjections, $\phi_{\half}$ is continuous and surjective onto $[-1,1]\times [0,1]$, and satisfies $\phi_{\half}(0)=(0,0)$.
We obtain the desired continuous odd surjection $s_2\colon [-1,1]\rightarrow [-1,1]^2$ by reflecting $\phi_{\half}$ onto the negative interval $[-1,0)$:
\[s_2(t)=\begin{cases}
    \phi_{\half}(t)  & \text{if } t \geq 0 \\
    -\phi_{\half}(-t)  & \text{if } t< 0
\end{cases}\]

\textbf{(ii) Inductive extension to arbitrary dimensions.}
We prove that for any $k\geq 2$, there exists a continuous odd surjection $s_k\colon [-1,1]\rightarrow [-1,1]^k$.
We proceed by induction. 
The base case $k=2$ is true by (i); let $s_2(t)=(p_1(t),p_2(t))$.
Assume a continuous odd surjection $s_k\colon [-1,1]\rightarrow [-1,1]^k$ exists.
Construct $s_{k+1}\colon [-1,1]\rightarrow [-1,1]^{k+1}$ via
$
s_{k+1}(t)=\left(p_1(t),s_k(p_2(t))\right).
$
Since $s_2$ and $s_k$ are each continuous odd surjections, it follows that $s_{k+1}$ is also a continuous odd surjection.

\textbf{(iii) Construction of a continuous odd surjection $[-1,1]^n\rightarrow [-1,1]^m$ for $m>n\geq 1$.}
For $m>n\geq 1$, let $s_m\colon [-1,1]\rightarrow [-1,1]^m$ be the continuous odd surjection from (ii). 
Define the surjective projection map $\pi\colon [-1,1]^n\rightarrow [-1,1]$ by $\pi(u_1,u_2,\ldots,u_n)=u_1$, which is continuous and odd.
Now we can define the continuous odd surjection $g\colon [-1,1]^n \rightarrow [-1,1]^m$ as the composition $g=s_m\circ \pi$.

\textbf{(iv) Euclidean balls and the boundary condition.}
Consider odd homeomorphisms $f_n\colon B^n\rightarrow [-1,1]^n$ and $f_m^{-1}\colon [-1,1]^m\rightarrow B^m$.
The map $\psi'_{m,n}=f_m^{-1}\circ g \circ f_n$ is now a continuous odd surjection $B^n \to B^m$ for $m>n\geq 1$.

Since $B^m$ is contractible, the set $[S^{n-1},B^m]$ of (unpointed) homotopy classes of continuous maps from $S^{n-1}$ to $B^m$ consists of a single homotopy class. 
Therefore, the restriction $\psi'_{m,n}\vert_{S^{n-1}}$ is homotopic to the standard inclusion $\iota_{m,n}\colon S^{n-1}\rightarrow B^m$, where $\iota(u_1,\ldots, u_n)=(u_1,\ldots, u_n, 0,\ldots,0)$.
Let $H': S^{n-1}\times I \to B^m$ be such a homotopy.
We can use $H'$ to construct an odd homotopy $H(u,t)=\frac{1}{2}\bigl(H'(u,t)-H'(-u,t)\bigr)$ with $H(u,0)=\psi'_{m,n}(u)$ and $H(u,1)=\iota_{m,n}(u)$.

Now, for $x=ru$, where $u\in S^{n-1}$ and $0\leq r\leq 1$, define
\[\psi_{m,n}(ru)=
\begin{cases}
    \psi'_{m,n}(2ru), &0\leq r\leq \frac{1}{2},\\
    H(u, 2r-1), & \frac{1}{2}\leq r \leq 1.
\end{cases}
\]
Note that $\psi_{m,n}\colon B^n\rightarrow B^m$ is a continuous odd surjection with $\psi_{m,n}|_{S^{n-1}}=\iota_{m,n}$, and hence $\psi_{m,n}(S^{n-1})\cap \{\vec{0}\}=\emptyset$.
\end{proof}

\section{Using curvature sets of balls to lower bound $d_\gh(B^m,B^n)$}
\label{sec:balls-curvature-sets}

We now explore an alternative approach to lower bounding $d_\gh(B^m, B^n)$ using the machinery of curvature sets, which were originally introduced by Gromov~\cite{gromov2007metric}.
For spheres $S^1$ and $S^2$ with geodesic metric, this approach yields $d_\gh(S^1,S^2)\geq \frac{\pi}{12}$ in~\cite{memoliGH}, with better bounds later obtained using different methods in~\cite{lim2023gromov}.
We investigate a similar technique with unit balls; the bounds we obtain here are weaker than those established in previous sections.

We begin by reviewing the necessary background on curvature sets and their relationship to the Gromov--Hausdorff distance.
Let $\mathcal{M}$ denote the collection of all compact metric spaces.

\begin{definition}[Curvature Sets]
\label{def:curvature-sets}
Let $(X,d) \in \mathcal{M}$ be a compact metric space and let $n \in \N$.
Let $\Psi^{(n)}_X\colon X^n \rightarrow \R^{n\times n}$ be the matrix-valued map $(x_1, \ldots, x_n)\mapsto ((d(x_i,x_j)))_{i,j=1}^n$.
Then, the \emph{$n$-th curvature set} of $X$ is
\[K_n(X)\coloneqq\{\Psi^{(n)}_X(x_1, \ldots , x_n)~:~(x_1, \ldots , x_n)\in X^n \}.\]

\end{definition}
Curvature sets contain all $n\times n$ distance matrices from ordered $n$-tuples of $X$, with repetitions allowed.

\begin{example}[Example of distance matrices]
\label{ex:distancematrices}
Let $M_{n+2}\in K_{n+2}(B^{n+1})$ be the distance matrix of the $n+2$ vertices of a regular $(n+1)$-simplex inscribed in $B^{n+1}$.
The diameter of the regular inscribed $(n+1)$-simplex in $S^n \subseteq B^{n+1}$ is $\sqrt\frac{2(n+2)}{n+1}$ (see for example~\cite{lovasz1983self}).
Two examples of $M_{n+2}\in K_{n+2}(B^{n+1})$ for $n=1,2$ are:
\[
  M_3=
  \left[ {\begin{array}{ccc}
   0 & \sqrt{3} & \sqrt{3}\\
   \sqrt{3} & 0 & \sqrt{3}\\
   \sqrt{3} & \sqrt{3} & 0\\
  \end{array} } \right]
  \quad
  M_4=
  \left[ {\begin{array}{cccc}
   0 & \sqrt{\frac{8}{3}} & \sqrt{\frac{8}{3}} & \sqrt{\frac{8}{3}} \\
   \sqrt{\frac{8}{3}} & 0 & \sqrt{\frac{8}{3}} & \sqrt{\frac{8}{3}}\\
   \sqrt{\frac{8}{3}} & \sqrt{\frac{8}{3}} & 0 & \sqrt{\frac{8}{3}}\\
 \sqrt{\frac{8}{3}} & \sqrt{\frac{8}{3}} & \sqrt{\frac{8}{3}} & 0\\
  \end{array} } \right]
  .
\]

\end{example}


Curvature sets $K_n(X)$ and $K_n(Y)$ are \emph{isometric invariants} of compact metric spaces $X$ and $Y$.
In fact, they are compact subsets of $\R^{n \times n}$, and so the Hausdorff distance between them is finite.
Let $\mathrm{Sym}_{k}^+$ denote the set of all symmetric $k \times k$
matrices with non-negative entries and zero diagonal, with metric $d_{\mathrm{Sym}_{k}^+}(A,B)=\max\limits_{i,j}\lvert a_{ij}-b_{ij}\rvert$ for $A=((a_{ij}))$ and $B=((b_{ij}))$ in $\mathrm{Sym}_{k}^+$.
One can relate the Gromov--Hausdorff distance to the Hausdorff distance between curvature sets for all $k\in\N$ as follows (see~\cite{memoliGH}):
\begin{equation}
\label{ineq:GH-curvature}
d_\gh(X,Y)\geq  \tfrac{1}{2}d_{\h}^{\mathrm{Sym}_{k}^+}(K_k(X),K_k(Y)), 
\end{equation}
where $d_{\h}^{\mathrm{Sym}_{k}^+}$ denotes the Hausdorff distance in $\mathrm{Sym}_{k}^+$.

In Theorem~\ref{thm:general} we prove that for every $m>n\ge 1$, we have $d_\gh(B^m,B^n) \ge \tfrac{1}{1+\alpha_n}$.
This provides stronger lower bounds, at least in the case $m=n+1$ for $n=1,2$, than the lower bound \emph{we were able to obtain} using curvature sets; see the following two remarks.
(It is conceivable that one could find improved lower bounds using curvature sets.)

Our strategy is to lower bound $\sup\limits_{i\in \N} d_\h(K_i(B^{n+1}),K_i(B^n))$ 
and then use~\eqref{ineq:GH-curvature} to obtain a lower bound on $d_{\gh}(B^{n+1},B^n)$.
We focus on $i=n+2$, the smallest $i$ such that $i$ points in $(n+1)$-dimensional space might not have a distance matrix arising from points in $n$-dimensional space.
Among $(n+2)$-point configurations in $B^{n+1}$, we consider $M_{n+2}$ (see Example~\ref{ex:distancematrices}), the distance matrix of the vertex set of a regular $(n+1)$-simplex  inscribed in $B^{n+1}$,  as this configuration maximizes symmetry and cannot be isometrically embedded in $\R^n$.
Using the fact that $M_{n+2}\in K_{n+2}(B^{n+1})$ and~\eqref{ineq:GH-curvature}, we have the following chain of inequalities: 

\[d_\gh(B^{n+1},B^n)\ge \tfrac{1}{2}d_{\h}(K_{n+2}(B^{n+1}),K_{n+2}(B^n)) \ge \tfrac{1}{2}d(M_{n+2},K_{n+2}(B^n)).\]

We now apply this lower bound explicitly for the low-dimensional cases $n=1,2$.

\begin{remark}
For $n=1$, we minimize the distance from $M_3$ (the distance matrix of a regular $2$-simplex in $B^2$) to all possible distance matrices in $K_{3}(B^1)$.
Placing three ordered points on the interval $[-1,1]$ at locations $-1,0,1$ yields: 
\[d_\gh(B^2,B^1)\ge \tfrac{1}{2}d(M_3,K_3(B^1))=\tfrac{\sqrt3-1}{2}\approx0.366.\]
\end{remark}

\begin{remark}
For $n=2$, we conjecture that the configuration that minimizes the distance to $M_4$ places points at $(\pm c,0)$ and $(0,\pm c)$ in $B^2$ for $c=\frac{2\sqrt\frac{8}{3}}{2+\sqrt{2}} \approx0.9566$, which (if true) would yield that $d_\gh(B^3,B^2)$ is at least as large as
$\tfrac{1}{2}d(M_4,K_4(B^2)){=}\tfrac{1}{2}\lvert\sqrt{\tfrac{8}{3}}-2c\rvert \approx0.1401$.
\end{remark}

Whether alternative optimization strategies or refined geometric configurations could improve these bounds remains an interesting open question.

\begin{question}

Does $d_\h(K_{n+2}(B^{n+1}),K_{n+2}(B^n))=d(M_{n+2},K_{n+2}(B^n))$?
\end{question}

\begin{question}
\label{ques2:curvaturesets}
Does $d_\h(K_{n+2}(B^{n+1}),K_{n+2}(B^n))\rightarrow0$ as $n\rightarrow\infty$?
\end{question}

\end{document}